\documentclass{article}

\usepackage{geometry}                
\usepackage{graphicx}
\usepackage{amssymb}
\usepackage{epstopdf}
\usepackage{pgf,tikz}
\usepackage{amsmath}
\usepackage{tikz}
\usepackage{graphicx}
\usepackage{mathtools}
\usepackage{graphicx}
\usepackage{epstopdf}
\usepackage{fix-cm}
\usepackage{geometry}                
\usepackage{graphicx}
\usepackage{amssymb}
\usepackage{epstopdf}
\usepackage{pgf,tikz}
\usepackage{latexsym}
\usepackage{amsmath,amsfonts}
\usepackage{graphicx}
\usepackage{geometry}
\usepackage{graphicx}
\usepackage{multirow}
\usepackage{amssymb}
\usepackage{epstopdf}
\usepackage{pgf,tikz}
\usepackage{latexsym}
\usepackage[english]{babel}
\usepackage{dcolumn}
\usepackage{color}

\usepackage{multirow}
\usepackage{amsmath,amsfonts,amsthm}
\usepackage{mathtools,url,tikz}
\usepackage{graphics}
\usepackage{graphicx}

\newtheorem{theorem}{Theorem}[section]

\newtheorem{corollary}[theorem]{Corollary}
\newtheorem{example}[theorem]{Example}
\newtheorem{lemma}[theorem]{Lemma}

\usepackage[colorlinks=true,breaklinks=true,bookmarks=true,urlcolor=blue,
     citecolor=blue,linkcolor=blue,bookmarksopen=false,draft=false]{hyperref}

\newcommand{\bK}{{\mathbf K}}
\newcommand{\oR}{{\mathbb R}}
\newcommand{\tfmax}{\widehat{f}_{\max}}

\newcommand{\fminK}{f_{\min,\mathbf{K}}}
\newcommand{\oN}{{\mathbb N}}
\newcommand{\vol}{\text{\rm vol}}

\begin{document}

\title{Comparison of Lasserre's measure--based  bounds for polynomial optimization to bounds obtained by simulated annealing}

\author{Etienne de Klerk \thanks{Tilburg University and Delft University of Technology, \texttt{E.deKlerk@uvt.nl}}
 \and Monique Laurent \thanks{Centrum Wiskunde \& Informatica (CWI), Amsterdam and Tilburg University, \texttt{monique@cwi.nl}}}

\maketitle

\begin{abstract}
We consider the problem of minimizing  a continuous function $f$  over a compact set $\bK$.
We compare the hierarchy of  upper bounds proposed by Lasserre in [{\em SIAM J. Optim.} $21(3)$ $(2011)$, pp. $864-885$]
to bounds that may be obtained from simulated annealing.

We show that, when $f$ is a polynomial and $\bK$ a convex body, this comparison yields a faster rate of convergence of the Lasserre hierarchy than
what was previously known in the literature. 
\end{abstract}


\section{Introduction }\label{secintro}

We consider the problem of minimizing a continuous function $f:\oR^n\to\oR$ over a compact set $\mathbf{K}\subseteq \oR^n$. That is, we consider the problem of computing the parameter:

\begin{equation*}\label{fmink}
f_{\min,\mathbf{K}}:= \min_{x\in \mathbf{K}}f(x).
\end{equation*}
Our goal is to compare two convergent hierarchies of upper bounds on $\fminK$, namely measure-based  bounds introduced by Lasserre \cite{Las11}, and simulated annealing bounds, as studied by Kalai and Vempala \cite{Kalai-Vempala 2006}.
The bounds of Lasserre are obtained by minimizing over measures on $\bK$ with sum-of-squares polynomial density functions with growing degrees, while simulated annealing bounds use Boltzman distributions on $\bK$ with decreasing temparature parameters.

In this note we establish a relationship between these two approaches, linking the degree and temperature parameters in the two bounds (see Theorem \ref{thm:main} for a precise statement).
As an application,  when $f$ is a polynomial and $K$ is a convex body,
 we can show a faster convergence rate for the measure-based bounds of Lasserre.
  The new convergence rate is in $O(1/r)$  (see Corollary~\ref{cor:nonconvex}),
  where $2r$ is  the degree of the sum-of-squares polynomial density function,
  while the dependence was in $O(1/\sqrt{r})$  in the previously best known result from \cite{KLS MPA}.

\medskip
Polynomial optimization is a very active research area in the recent years since the seminal works of Lasserre \cite{Las01}
and Parrilo \cite{Par} (see also, e.g., the book \cite{Las09} and the survey \cite{Lau09}). In particular, hierarchies
 of (lower and upper) bounds for the parameter $\fminK$ have been proposed, based on sum-of-squares  polynomials and semidefinite programming.

For a general compact set $\bK$, upper bounds for $\fminK$ have been introduced by Lasserre \cite{Las11}, obtained by searching for a sum-of-squares polynomial density function of given maximum degree $2r$, so as to minimize the integration of $f$ with respect to the corresponding probability measure on $\bK$.
When $f$ is Lipschitz continuous and under some mild assumption on $\bK$ (which holds, e.g., when $\bK$ is a convex body), estimates for the convergence rate of these bounds have been proved in \cite{KLS MPA} that are in order $O(1/\sqrt r)$. Improved rates have been subsequently shown when restricting to special sets $\bK$. Related stronger results have been shown for the case when $\bK$ is the hypercube $[0,1]^n$ or $[-1,1]^n$.
In \cite{KLLS MOR} the authors show a hierarchy of upper bounds using the Beta distribution, with the same convergence rate in $O(1/\sqrt{r})$,
but whose computation needs only elementary operations; moreover an improved convergence in $O(1/r)$ can be shown, e.g.,  when $f$ is quadratic.
In addition, a convergence rate in $O(1/r^2)$ is shown in \cite{DHL SIOPT}, using distributions based on Jackson kernels
and a larger class of sum-of-squares density functions.


In this paper we investigate the hierarchy of measure-based upper bounds of \cite{Las11} and show that when $K$ is a convex body,  convexity can be exploited to show an improved convergence rate in $O(1/r)$, even for nonconvex functions.
  The key ingredient for this is to establish a relationship with upper bounds based on simulated annealing and to
  use a known  convergence rate result from \cite{Kalai-Vempala 2006} for simulated annealing bounds in the convex case.

\medskip
Simulated annealing was introduced by Kirkpatrick et al. \cite{Kirkpatrick et al 1983} as a randomized search procedure for general optimization
problems.
 It has enjoyed renewed interest  for convex optimization problems since it was shown
by Kalai and Vempala \cite{Kalai-Vempala 2006} that a polynomial-time implementation is possible.
This requires so-called hit-and-run sampling from $\mathbf{K}$, as
introduced by Smith \cite{Smith 1984}, that was shown to be a polynomial-time procedure by Lov\'asz \cite{Lovasz1999}.
Most recently, Abernethy and Hazan \cite{Abernethy_Hazan_2015} showed formal equivalence with a certain interior point
method for convex optimization.

This unexpected equivalence between seemingly different methods has motivated this current work to relate the bounds by Lasserre \cite{Las11}
 to the simulating annealing bounds as well.

\medskip
In what follows, we first introduce the measure-based upper bounds of Lasserre~\cite{Las11}. Then we recall the bounds based on simulated annealing and the known convergence results for a linear objective function $f$, and we give an explicit proof of their extension to the case of a general convex function $f$. After that we state our main result and
the next section is devoted to its proof.
In the last section we conclude with numerical examples showing the quality of the two types of bounds and some final remarks.

\section{Lasserre's hierarchy of upper bounds}
Throughout,  $\oR[x]=\oR[x_1,\dots,x_n]$ is the set of polynomials in $n$ variables with real coefficients and, for an integer $r\in \oN$, $\oR[x]_r$ is the set of polynomials with degree at most $r$. Any polynomial $f\in \oR[x]_r$ can be written $f=\sum_{\alpha\in N(n,r)} f_\alpha x^\alpha$, where we set $x^\alpha=\prod_{i=1}^nx_i^{\alpha_i}$ for $\alpha\in \oN^n$ and  $N(n,r)=\{\alpha\in \oN^n: \sum_{i=1}^n\alpha_i\le r\}$.
We let $\Sigma[x]$ denote the set of sums of squares of polynomials, and $\Sigma[x]_r=\Sigma[x]\cap \oR[x]_{2r}$ consists of all  sums of squares of polynomials with degree at most $2r$. 

\medskip We recall the following reformulation for $f_{\min,\mathbf{K}}$, established by Lasserre \cite{Las11}:

\begin{equation*}\label{fminkreform2}
f_{\min,\mathbf{K}}=\inf_{h\in\Sigma[x]}\int_{\mathbf{K}}h(x)f(x)dx \ \ \mbox{s.t. $\int_{\mathbf{K}}h(x)dx=1$.}
\end{equation*}


\smallskip
\noindent
By bounding the degree of the polynomial $h\in \Sigma[x]$ by $2r$, we can define  the parameter:

\begin{eqnarray}\label{fundr}
\underline{f}^{(r)}_{\mathbf{K}}:=\inf_{h\in\Sigma[x]_r}\int_{\mathbf{K}}h(x)f(x)dx \ \ \mbox{s.t. $\int_{\mathbf{K}}h(x)dx=1$.}
\end{eqnarray}

\smallskip
\noindent
Clearly,  the inequality  $f_{\min,\mathbf{K}}\le\underline{f}^{(r)}_{\mathbf{K}}$ holds for all $r\in\oN$.  Lasserre \cite{Las11} gave conditions under which  the infimum is attained in the program (\ref{fundr}). De Klerk, Laurent and Sun \cite[Theorem 3]{KLS MPA} established the following rate of convergence for the bounds
$\underline{f}^{(r)}_{\mathbf{K}}$.

\begin{theorem}[De Klerk, Laurent, and Sun \cite{KLS MPA}] \label{thm:dKLS}
Let $f\in \oR[x]$ and  $\mathbf{K}$ a convex body. There exist constants $C_{f,\bK}$ (depending only on $f$ and $\bK$) and $r_\bK$ (depending only on $\bK$) such that
\begin{equation}\label{thmmaineq2}
\underline{f}^{(r)}_{\mathbf{K}}-f_{\min,\mathbf{K}} \le {C_{f,\bK} \over \sqrt {r}}\ \ \text{ for all } r\ge r_\bK.
\end{equation}
 That is, the  following asymptotic convergence rate holds: $\underline{f}^{(r)}_{\mathbf{K}}-f_{\min,\mathbf{K}} \simeq O\left( {1\over \sqrt r}\right).$
\end{theorem}
This result of \cite{KLS MPA} holds in fact under more general assumptions, namely when $f$ is Lipschitz continuous and $\bK$ satisfies a technical assumption (Assumption 1 in \cite{KLS MPA}), which says (roughly) that around any point in $ \mathbf K$ there is a ball whose intersection with $\bK$ is at least  a constant fraction of the unit ball.

As explained in \cite{Las11}  the parameter $\underline{f}^{(r)}_{\mathbf{K}}$ can be computed using semidefinite programming, assuming  one knows  the moments $m_\alpha(\bK)$ of the Lebesgue measure on $\bK$, where
\begin{equation*}\label{mack}
m_{\alpha}(\mathbf{K}):=\int_{\mathbf{K}}x^{\alpha}dx\ \ \ \mbox{ for } \alpha\in \oN^n.
\end{equation*}
\smallskip
\noindent
Indeed suppose $f(x)=\sum_{\beta\in N(n,d)}f_{\beta}x^{\beta}$ has degree $d$.
Writing  $h\in\Sigma[x]_{r}$ as $h(x)=\sum_{\alpha\in N(n,2r)}h_{\alpha}x^{\alpha}$,  the parameter $\underline{f}^{(r)}_{\mathbf{K}}$ from
 (\ref{fundr}) can be reformulated as follows:
\begin{eqnarray}\label{eqSDP}
\underline{f}^{(r)}_{\mathbf{K}}&=&\min\sum_{\beta\in N(n,d)}f_{\beta}\sum_{\alpha\in N(n,2r)}h_{\alpha}m_{\alpha+\beta}(\mathbf{K})\label{fundr2}\\
 & &\mbox{ s.t. } \ \ \sum_{\alpha\in N(n,2r)}h_{\alpha}m_{\alpha}(\mathbf{K})=1,\nonumber\\
&&\ \ \ \ \ \ \ \sum_{\alpha\in N(n,2r)}h_{\alpha}x^{\alpha}\in\Sigma[x]_r.\nonumber
\end{eqnarray}
Since the sum-of-squares condition on $h$ may be written as a linear matrix inequality, this is a semidefinite program.
In fact, since it only has one linear equality constraint, it may even be rewritten as a generalised eigenvalue problem.
In particular,
  $\underline{f}_{\mathbf{K}}^{(r)}$ {is equal to the   the smallest  generalized eigenvalue} of the system:
\[
Ax = \lambda Bx \quad \quad\quad (x \neq 0),
\]
where the symmetric matrices $A$ and $B$ are of order ${n + r \choose r}$ with rows and columns  indexed by $N(n,r)$,
and
\begin{equation}
\label{matrices A and B}
A_{\alpha, \beta} = \sum_{\delta \in N(n,d)} f_\delta \int_{\mathbf{K}} x^{\alpha + \beta + \delta} dx, \quad B_{\alpha, \beta} = \int_{\mathbf{K}} x^{\alpha + \beta} dx \quad \alpha, \beta \in {N}(n,r).
\end{equation}
For more details, see \cite{Las11,KLS MPA,KLLS MOR}.

\section{Bounds from simulated annealing}

Given a continuous function $f$, consider the associated Boltzman distribution over the set $\bK$, defined by the density function:
\[
\textstyle
	P_f(x) := \frac{\exp(-f(x))}
	{ \int_{\bK} \exp(-f(x') ) \, dx' }.
\]
Write
$
X \sim P_f
$
if the random variable $X$ takes values in $\bK$ according to the Boltzman distribution.

The idea of simulated annealing is to sample $X \sim P_{f/t}$ where $t > 0$ is a fixed `temperature' parameter, that is subsequently decreased.
Clearly, for any $t>0$,  we have
\begin{equation}\label{eqBoltz}
\fminK \le  \mathbb{E}_{X \sim P_{f/t}}[f(X)].
\end{equation}
The point is that, under mild assumptions, these bounds converge to the minimum of $f$ over $\bK$ (see, e.g., \cite{Spall}):
\[
\lim_{t \downarrow 0} \mathbb{E}_{X \sim P_{f/t}}[f(X)] = f_{\min,\mathbf{K}}.
\]
The key step in the practical utilization of theses bounds is therefore to perform the sampling of $X \sim P_{f/t}$.

\medskip
\begin{example}
\label{ex:1}
Consider the minimization of the Motzkin polynomial $$f(x_1,x_2)=64(x_1^4x_2^2+x_1^2x_2^4) - 48x_1^2x_2^2 +1$$ over  $\bK = [-1,1]^2$,
where there are four global minimizers at the points
$\left(\pm \frac{1}{2},\pm \frac{1}{2}\right)$, and $f_{\min,\bK} = 0$.
Figure \ref{figure:motzkin_SA} shows the
corresponding Boltzman  density function for $t = \frac{1}{2}$.
Note that this  density has four modes, roughly positioned at the four global minimizers of $f$ in $[-1,1]^2$.
The corresponding upper bound on $f_{\min,\bK} = 0$ is $\mathbb{E}_{X \sim P_{f/t}}[f(X)] \approx 0.7257$ ($t = \frac{1}{2}$).

\begin{figure}[h!]
\begin{center}
\includegraphics[width=0.45\textwidth]{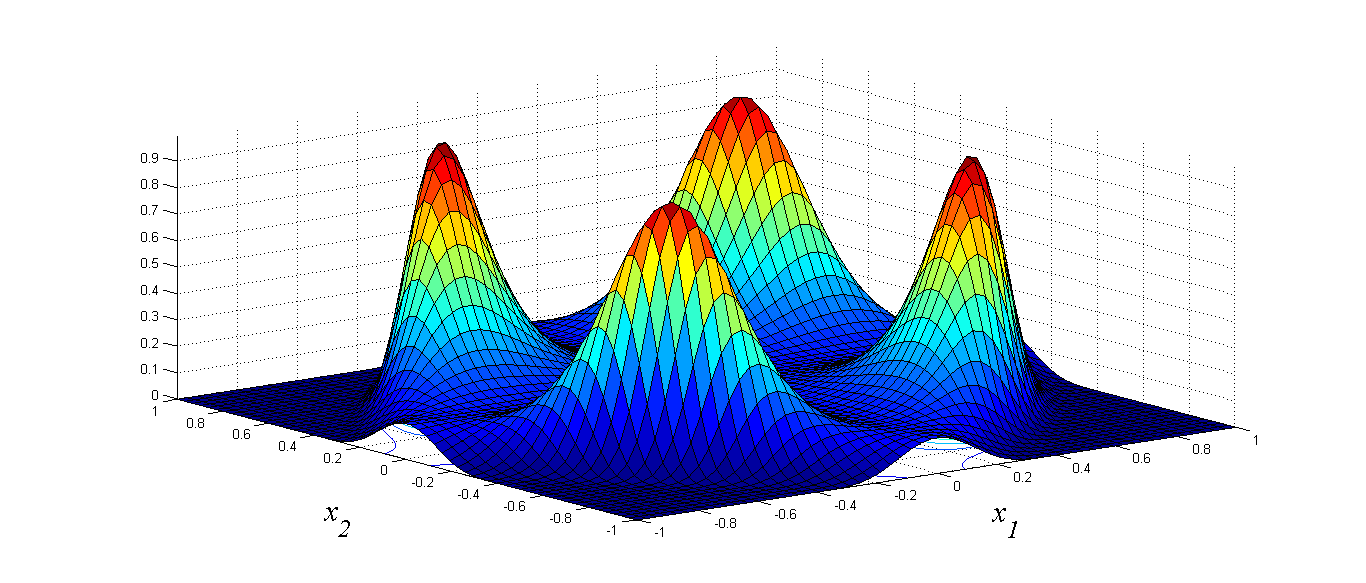}
\includegraphics[width=0.45\textwidth]{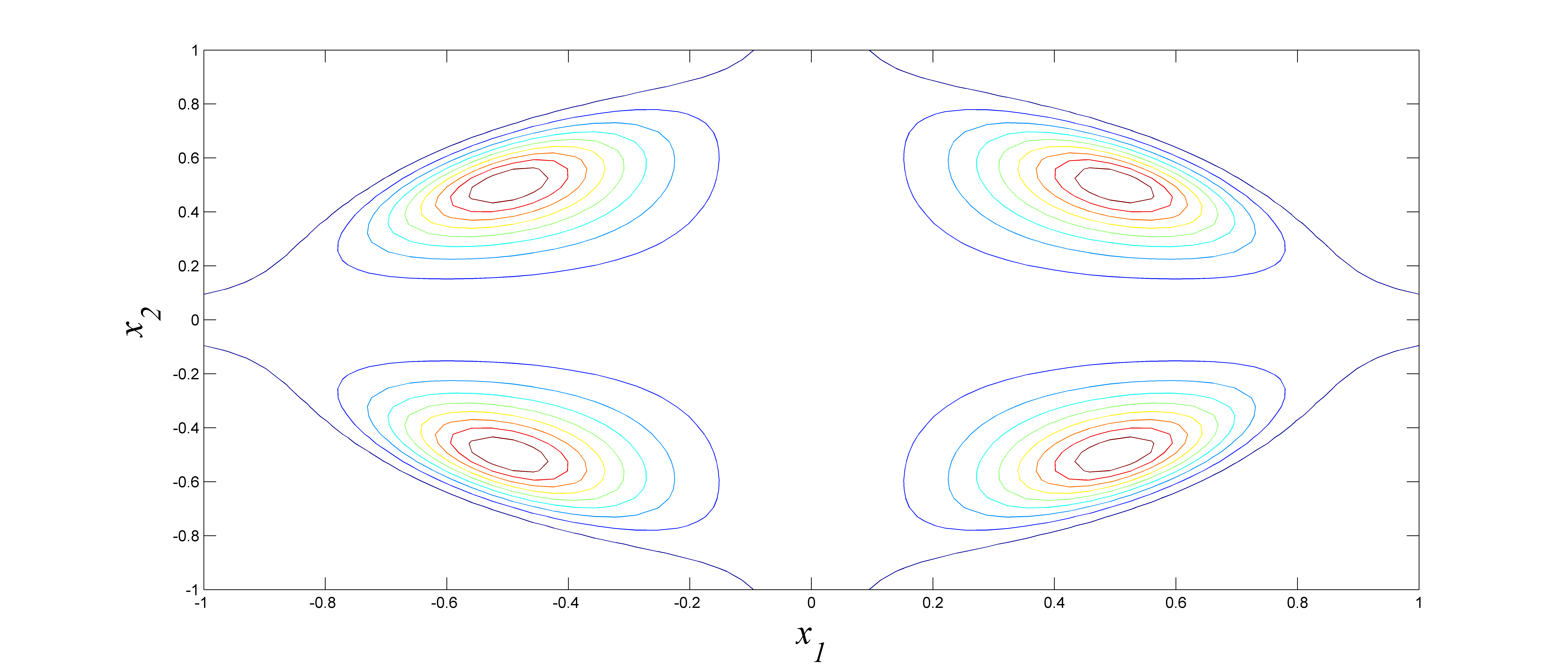}\\
  \caption{\label{figure:motzkin_SA}Graph and contours of the Boltzman density with $t = \frac{1}{2}$ for the Motzkin polynomial.}
\end{center}

\end{figure}

To obtain a better upper bound on $f_{\min,\bK}$ from the Lasserre hierarchy, one needs to
use a degree $14$ s.o.s.\ polynomial density; in particular, one has $\underline{f}^{(6)}_{\mathbf{K}}=0.8010$ (degree $12$) and $\underline{f}^{(7)}_{\mathbf{K}}= 0.7088$ (degree $14$). More detailed numerical results are given in Section \ref{sec:conclusion}.
\end{example}

\medskip
When $f$ is linear and $\bK$ a convex body, Kalai and Vempala \cite[Lemma 4.1]{Kalai-Vempala 2006}  show that the  rate of convergence of the bounds in (\ref{eqBoltz}) is linear in the temperature $t$.

\begin{theorem}[Kalai and Vempala \cite{Kalai-Vempala 2006}]
\label{thm:Kallai-Vempala}
Let  $f(x) = c^Tx$ where $c$ is a unit vector,
and let $\bK$ be a convex body. Then, for any $t>0$,  we have
\[
	 \mathop{\mathbb{E}}_{X \sim P_{f/t  }}[f(X)] - \min_{x \in \bK} f(x) \leq  n t.
\]
\end{theorem}

We indicate how to extend the result of Kalai and Vempala in Theorem \ref{thm:Kallai-Vempala} to the case of an arbitrary convex function $f$.
This more general result is hinted at in \S6 of \cite{Kalai-Vempala 2006}, where the authors write
\begin{quote}
``... a statement
analogous to [Theorem 2] holds also for general convex functions ..."
\end{quote}
 but no precise statement is given there. In any event, as we will now show, the more general result may readily be derived from Theorem~\ref{thm:Kallai-Vempala} (in fact, from the special case of a linear coordinate function $f(x)=x_i$ for some $i$).

\begin{corollary}\label{corconvex}
Let $f$ be a convex function and let $\bK\subseteq \oR^n$ be a convex body.
Then, for any $t>0$, we have
\[
	 \mathop{\mathbb{E}}_{X \sim P_{f/t  }}[f(X)] - \min_{x \in \bK} f(x) \leq  n t.
\]
\end{corollary}

\begin{proof}
Set
 $$E_\bK:= \mathop{\mathbb{E}}_{X \sim P_{f/t  }}[f(X)]
= {\int_\bK f(x)e^{-f(x)/t} dx\over \int_\bK e^{-f(x)\over t} dx}.$$
Then  we have $$\fminK=\min_{x\in \bK} f(x)\le E_\bK.$$
Define the set
$$\widehat \bK:=\{(x,x_{n+1})\in \oR^{n+1}: x\in \bK,\ f(x)\le x_{n+1}\le E_\bK\}.$$
Then $\widehat \bK$ is a convex body and we have
$$\min_{x\in \bK}f(x)=\min_{(x,x_{n+1})\in \widehat \bK}  x_{n+1}.$$
Accordingly, define the parameter
$$E_{\widehat \bK}:= {\int_{\widehat \bK} x_{n+1}e^{-x_{n+1}/t} dx_{n+1}dx \over \int_{\widehat \bK} e^{-x_{n+1}/t} dx_{n+1}dx}.$$
Corollary \ref{corconvex} will follow if we show that 
\begin{equation}\label{eqEE}
E_{\widehat\bK}=E_{ \bK}+t.
\end{equation}
To this end set   $E_\bK={N_\bK\over D_\bK}$ and $E_{\widehat \bK}={N_{\widehat \bK}\over D_{\widehat \bK}}$, where we define
$$N_\bK:= \int_\bK f(x)e^{-f(x)/t} dx,\ \
D_\bK:=\int_\bK e^{-f(x)/t}dx,$$
$$N_{\widehat \bK}:= \int_{\widehat \bK} x_{n+1}e^{-x_{n+1}/t} dx_{n+1}dx,\ \
D_{\widehat \bK}:=  \int_{\widehat \bK} e^{-x_{n+1}/t} dx_{n+1}dx.$$
We work out the parameters $N_{\widehat \bK}$ and $D_{\widehat \bK}$ (taking integrations by part):
$$D_{\widehat \bK}= \int_\bK \left(\int_{f(x)}^{E_\bK} e^{-x_{n+1}/t}dx_{n+1}\right) dx = \int_\bK\left(te^{-f(x)/t} -t e^{-E_\bK/t}\right)dx = tD_\bK -t e^{-E_\bK/t}\vol(\bK),$$
\begin{eqnarray*}
N_{\widehat \bK}   &=& \int_\bK \left(\int_{f(x)}^{E_\bK} x_{n+1} e^{-x_{n+1}/t}dx_{n+1}\right) dx \\
&= & \int_\bK \left(-tE_\bK e^{-E_\bK/t}+t f(x)e^{-f(x)/t} +t \int_{f(x)}^{E_\bK} e^{-x_{n+1}/t} dx_{n+1}\right) dx \\
&=& -tE_\bK e^{-E_\bK/t}\vol(\bK) +t N_\bK +t D_{\widehat \bK}.
\end{eqnarray*}
Then, using the fact that $E_\bK={N_\bK\over D_\bK}$,  we obtain:
$${N_{\widehat \bK}\over D_{\widehat \bK}}= t + {N_\bK-E_\bK e^{-E_\bK/t}\vol(K)\over D_\bK -e^{-E_\bK/t}\vol(K)} = t +{N_\bK\over D_\bK},
$$
which proves relation (\ref{eqEE}).

We can now derive the result of Corollary \ref{corconvex}.
Indeed, using Theorem 2 applied to $\widehat \bK$ and the linear function $x_{n+1}$, we get
$$\mathop{\mathbb{E}}_{X \sim P_{f/t  }}[f(X)] - \min_{x \in \bK} f(x) = E_K - \min_{x\in\bK} f(x) = (E_{\widehat \bK} -\min_{(x,x_{n+1})\in \widehat \bK} x_{n+1}) + (E_{\bK} -E_{\widehat \bK}) \le t(n+1) -t =tn.$$
\mbox{}\hfill\qed
\end{proof}

\medskip

The bound in the corollary is tight asymptotically, as the following example shows.

\medskip
\begin{example}
\label{ex:tight bound}
Consider  the univariate problem $\min_x \{x \; | \; x \in [0,1]\}$.
Thus, in this case, $f(x) = x$, $\bK = [0,1]$ and  $\min_{x \in \bK} f(x)=0$.
For given temperature $t>0$, we have
\begin{eqnarray*}
 \mathop{\mathbb{E}}_{X \sim P_{f/t  }}[f(X)] - \min_{x \in \bK} f(x) = \frac{\int_0^1 x e^{-x/t}dx}{\int_0^\ell  e^{-x/t}dx} - 0
 = t-{e^{-1/t}\over 1-e^{-1/t}}  \sim  t \ \mbox{ for small } t.\\
  \end{eqnarray*}
\end{example}

\section{Main results}
We will prove the following relationship between the sum-of-squares based upper bound (\ref{fundr}) of Lasserre and the bound  (\ref{eqBoltz})
based on simulated annealing.

\begin{theorem}
\label{thm:main}
Let $f$ be a polynomial of degree $d$,  let  $\bK$ be a compact set and set $\tfmax=\max_{x\in \bK} |f(x)|.$ Then we have
\[
\underline{f}^{(rd)}_{\mathbf{K}} \le  \mathop{\mathbb{E}}_{X \sim P_{f/t  }}[f(X)]  + {\tfmax\over 2^r} \ \ \mbox{ for any integer } \  r \ge {e\cdot \tfmax \over t}\ \mbox{ and any } \ t>0.
\]
\end{theorem}

For the problem of  minimizing a convex polynomial  function over a convex body, we obtain the following improved convergence rate for the sum-of-squares based bounds of Lasserre.

\begin{corollary}\label{cor:main}
Let $f\in \oR[x]$ be a convex polynomial  of degree $d$ and let $\bK$ be a convex body. Then for any integer $r\ge 1$ one has
\[
	 \underline{f}^{(rd)}_{\mathbf{K}} - \min_{x \in \bK} f(x) \leq  \frac{c}{r},
\]
for some constant $c>0$ that does not depend on $r$.
(For instance, $c=(ne+1)\tfmax$.)
\end{corollary}

\begin{proof}
Let $r\ge 1$ and set $t={e\cdot \tfmax\over r}$. Combining Theorems \ref{thm:Kallai-Vempala} and \ref{thm:main}, we get
\begin{eqnarray*}
 \underline{f}^{(rd)}_{\mathbf{K}} - \min_{x \in \bK} f(x) &=&
\big( \underline{f}^{(rd)}_{\mathbf{K}} -  \mathop{\mathbb{E}}_{X \sim P_{f/t  }}[f(X)] \big)
+\big( \mathop{\mathbb{E}}_{X \sim P_{f/t  }}[f(X)] -\fminK\big) \\
&\leq& {\tfmax\over 2^r}+nt = {\tfmax\over 2^r} + {ne\cdot \tfmax\over r}  \le {(ne+1)\tfmax\over r}.
\end{eqnarray*}
\hfill\qed\end{proof}

For  convex polynomials $f$, this improves on the known $O(1/\sqrt{r})$ result from Theorem \ref{thm:dKLS}. 
One may in fact use the last corollary to obtain the same rate of convergence in terms of $r$ for all polynomials, without the
convexity assumption, as we will now show.

\begin{corollary}
\label{cor:nonconvex}
If $f$ be a polynomial and $\bK$ a convex body, then there is a $c > 0$ depending on $f$ and $\bK$ only, so that
\[
\underline{f}^{(2r)}_{\mathbf{K}} - \min_{x \in \bK} f(x) \le \frac{c}{r}.
\]
A suitable value for $c$ is
\[
c = (ne+1)\left(f_{\min,\bK} +C^1_f\cdot \mbox{diam}(\bK) + C^2_f \cdot \mbox{diam}(\bK)^2\right),
\]
where
$C^1_f = \max_{x \in \bK} \| \nabla f(x) \|_2$ and
$C^2_f = \max_{x \in \bK} \| \nabla^2 f(x) \|_2$.
\end{corollary}
We first define a convex quadratic function $q$ that upper bounds $f$ on $\bK$ as follows:
\[
q(x) = f(a) + \nabla f(a)^\top (x-a) + C^2_f \|x-a\|_2^2,
\]
where $C^2_f = \max_{x \in \bK} \| \nabla^2 f(x) \|_2$, and $a$ is the minimizer of $f$ on $\bK$.
Note that $q(x) \ge f(x)$ for all $ x \in \bK$ by Taylor's theorem, and
$\min_{x \in \bK} q(x) = f(a)$.

By definition of the Lasserre hierarchy,
\begin{eqnarray*}
\underline{f}^{(2r)}_{\mathbf{K}} &:=&\inf_{h\in\Sigma[x]_{2r}}\int_{\mathbf{K}}h(x)f(x)dx \ \ \mbox{s.t. $\int_{\mathbf{K}}h(x)dx=1$}\\
           &\le&\inf_{h\in\Sigma[x]_{2r}}\int_{\mathbf{K}}h(x)q(x)dx \ \ \mbox{s.t. $\int_{\mathbf{K}}h(x)dx=1$} \\
           &\equiv& \underline{q}_{\bK}^{(2r)}.
\end{eqnarray*}
Invoking Corollary \ref{cor:main} and using that the degree of $q$ is $2$, we obtain:
\[
\underline{f}^{(2r)}_{\mathbf{K}} \le \underline{q}_{\bK}^{(2r)} \le f(a) + \frac{(ne+1)\hat q_{\max}}{r},
  \]
where $\hat q_{\max} = \max_{x \in \bK} q(x)\le f_{\min,\bK} +C^1_f\cdot \mbox{diam}(\bK) + C^2_f \cdot \mbox{diam}(\bK)^2$.\qed

The last result improves on the known $O\left(\frac{1}{\sqrt{r}} \right)$ rate in Theorem \ref{thm:dKLS}.

\subsection*{Proof of Theorem \ref{thm:main}}

The key idea in the proof of Theorem \ref{thm:main} is to replace the Boltzman density function by a polynomial approximation.

To this end, we first recall a basic result on approximating the exponential function by its truncated Taylor series.

\begin{lemma}[De Klerk, Laurent and Sun \cite{KLS MPA}]\label{lemf2rsos}
Let $\phi_{2r}(\lambda)$ denote the (univariate) polynomial of degree $2r$ obtained by truncating the Taylor series expansion of $e^{-\lambda}$  at the order $2r$. That is,
\begin{equation*}\label{phi2rt}
\phi_{2r}(\lambda):=\sum_{k=0}^{2r}{(-t)^k\over k!}.
\end{equation*}
Then  $\phi_{2r}$ is a sum of squares of polynomials.
Moreover, we have
\begin{equation}\label{fmf2r}
0\le \phi_{2r}(\lambda) - e^{-\lambda} \le {\lambda^{2r+1}\over (2r+1)!} \quad \mbox{ for all } \lambda\ge 0.
\end{equation}
\end{lemma}

We now define the following approximation of the Boltzman density $P_{f/t}$:
\begin{equation}\label{eq:density}
\varphi_{2r,t}(x) := \frac{\phi_{2r}(f(x)/t)}{\int_{\bK} \phi_{2r}(f(x)/t)dx}.
\end{equation}
By construction,  $\varphi_{2r,t}$ is a sum-of-squares polynomial probability density function on $\bK$, with degree $2rd$ if $f$ is a polynomial of degree $d$.
Moreover, by relation \eqref{fmf2r} in Lemma \ref{lemf2rsos}, we obtain
\begin{eqnarray}
\varphi_{2r,t}(x) & \le & \frac{\phi_{2r}(f(x)/t)}{\int_{\bK} \exp(-f(x)/t)dx}  \\
               & \le & P_{f/t}(x) + \frac{ {(f(x)/t)^{2r+1} }}{(2r+1)!\int_{\bK} \exp(-f(x)/t)dx}.\label{eqphi}
\end{eqnarray}
From this we can derive  the following result.

\begin{lemma}
\label{lemma:err}
For any continuous $f$ and scalar $t>0$ one has
\begin{equation}
\label{eq:last term}
\underline{f}^{(rd)}_{\mathbf{K}} \le\int_{\bK}f(x)\varphi_{2r,t}(x)dx \le \mathop{\mathbb{E}}_{X \sim P_{f/t  }}[f(X)] +
 \frac{\int_{\bK}(f(x)-\fminK) (f(x))^{2r+1}dx}{t^{2r+1}(2r+1)!\int_{\bK} \exp(-f(x)/t)dx}.
\end{equation}
\end{lemma}

\begin{proof}
As $\varphi_{2r,t}(x)$ is a polynomial of degree $2rd$ and a probability density function on $\bK$ (by (\ref{eq:density})), we have:
\begin{equation}\label{eq0}
\underline{f}^{(rd)}_{\mathbf{K}} \le \int_\bK f(x)\varphi_{2r,t}(x) dx =\int_\bK (f(x)-\fminK) \varphi_{2r,t}(x)dx + \fminK.
\end{equation}
Using the above inequality (\ref{eqphi}) for $\varphi_{2r,t}(x)$ we can upper bound the integral on the right hand side:
\begin{eqnarray*}
\int_K (f(x)-\fminK) \varphi_{2r,t}(x)dx & \le
& \int_\bK (f(x)-\fminK)P_{f/t}(x)dx +
\int_\bK { (f(x)-\fminK) (f(x)/t)^{2r+1}\over (2r+1)! \int_K \exp(-f(x)/t)dx } dx\\
& = &  \mathop{\mathbb{E}}_{X \sim P_{f/t  }}[f(X)] - \fminK + \int_\bK { (f(x)-f_{\min}) (f(x)/t)^{2r+1}\over (2r+1)! \int_K \exp(-f(x)/t)dx } dx.
\end{eqnarray*}
Combining with the inequality (\ref{eq0}) gives the desired result.\qed
\end{proof}

\medskip
We now proceed to the proof of Theorem \ref{thm:main}. In view of Lemma \ref{lemma:err}, we only need to bound the last right-hand-side term in \eqref{eq:last term}:
$$T:=  \frac{\int_{\bK}(f(x)-f_{\min}) (f(x))^{2r+1}dx}{t^{2r+1}(2r+1)!\int_{\bK} \exp(-f(x)/t)dx}$$
and to show that $T\le {\tfmax\over 2^r}$.

By the defininition of $\tfmax$ we have
$$(f(x)-f_{\min})(f(x))^{2r+1}\le 2 \tfmax^{2(r+1)}\ \mbox{ and } \   \exp(-f(x)/t) \ge \exp(\tfmax/t) \ \mbox{ on } \bK,$$
which implies
$$ T \le {2 \tfmax^{2(r+1)} \exp(\tfmax/t) \over t^{2r+1}(2r+1)!}.$$
Combining with the
Stirling approximation inequality,
\[
r! \ge \sqrt{2\pi r}\left(\frac{r}{e}\right)^r  \quad\quad\quad (r \in \mathbb{N}),
\]
 applied to $(2r+1)!$, we obtain:
$$T
\le
{2\tfmax\over \sqrt{2\pi(2r+1)} } \left({\tfmax e \over t(2r+1)}\right)^{2r+1} \exp(\tfmax/t).$$
Consider   $r\ge {e\cdot \tfmax \over t}$, so that $\tfmax/t \le r/e$. Then, using the fact that $r/(2r+1)\le 1/2$,  we obtain
\begin{eqnarray*}
T &\le&
\frac{2\tfmax}{\sqrt{2\pi}}\frac{\exp(r/e)}{\sqrt{2r+1}}\left(\frac{r}{2r+1}\right)^{2r+1} \\
&\le&
\frac{\tfmax}{\sqrt{2\pi}}\frac{\exp(1/e)^r}{\sqrt{2r+1}}\left(\frac{1}{4}\right)^{r} \\
&=&
\frac{\tfmax}{\sqrt{2\pi}\sqrt{2r+1}}\left(\frac{\exp(1/e)}{4}\right)^r\\
&< & {\tfmax \over 2^r}.\\
\end{eqnarray*}
This concludes the proof of Theorem \ref{thm:main}.


\section{Concluding remarks}
\label{sec:conclusion}
We conclude with a numerical comparison of  the two hierarchies of bounds.
By Theorem \ref{thm:main}, it is reasonable to compare the bounds $\underline{f}^{(r)}_{\mathbf{K}}$ and
$\mathop{\mathbb{E}}_{X \sim P_{f/t  }}[f(X)]$, with $t = \frac{e\cdot d \cdot \tfmax}{r}$ and $d$ the degree of $f$.
Thus we define, for the purpose of comparison:
\[
SA^{(r)} = \mathop{\mathbb{E}}_{X \sim P_{f/t  }}[f(X)], \mbox{ with $t = \frac{e\cdot d \cdot \tfmax}{r}$}.
\]

We calculated the bounds for the polynomial test functions listed in Table \ref{tab:test}.

\begin{table}[h!]
\caption{Test functions, all with $n=2$, domain $\bK = [-1,1]^2$, and minimum $f_{\min,\bK} = 0$. \label{tab:test}}
\begin{center}
\begin{tabular}{|c|c|c|c|c|}\hline
Name & $f(x)$ & $\tfmax$ & $d$ & Convex?\\ \hline
Booth function& $(10x_1+20x_2-7)^2 + (20x_1+10x_2-5)^2$ & $2594$ & $2$ & yes \\ \hline
Matyas function& $26(x_1^2+x_2^2)-48x_1x_2$ & $100$ & $2$ & yes \\ \hline
Motzkin polynomial& $64(x_1^4x_2^2+x_1^2x_2^4) - 48x_1^2x_2^2 +1$ & $81$ & $6$ & no \\ \hline
Three-Hump Camel function& $\frac{5^6}{6}x_1^6-5^4\cdot 1.05x_1^4+50x_1^2+25x_1x_2+25x_2^2$ & $2048$ & $6$ &  no \\ \hline
\end{tabular}
 \end{center}
\end{table}

The bounds are shown in Table \ref{table:result1}.
The bounds $\underline{f}_{\mathbf{K}}^{(r)}$ were taken from \cite{DHL SIOPT},
while the bounds $SA^{(r)}$ were computed via numerical integration, in particular using the
Matlab routine {\tt sum2} of the package Chebfun \cite{chebfun}.
\begin{table}[h!]
\caption{Comparison of the upper bounds $SA^{(r)}$ and $\underline{f}_{\mathbf{K}}^{(r)}$ for the test functions.
\label{table:result1}}
\begin{center}
{\small
	\begin{tabular}{| c | >{\centering}m{1.2cm} | c | >{\centering}m{1cm} | c | >{\centering}m{1.3cm} | c | >{\centering}m{1.2cm} | c |}
    \hline
\multirow{2}{*}{{$r$}} & \multicolumn{2}{c|}{Booth Function} & \multicolumn{2}{c|}{Matyas Function} & \multicolumn{2}{m{3cm}|}{\centering Three--Hump Camel Function}& \multicolumn{2}{c|}{Motzkin Polynomial} \\ \cline{2-9}
                   & $\underline{f}_{\mathbf{K}}^{(r)}$  & $SA^{(r)}$& $\underline{f}_{\mathbf{K}}^{(r)}$  & $SA^{(r)}$  & $\underline{f}_{\mathbf{K}}^{(r)}$  & $SA^{(r)}$  & $\underline{f}_{\mathbf{K}}^{(r)}$  & $SA^{(r)}$\\ \hline
$3$  & 118.383 & 367.834   & 4.2817 &15.4212  & 29.0005&247.462   & 1.0614   &   4.0250\\ \hline
$4$  & 97.6473 & 356.113   & 3.8942 &14.8521  & 9.5806 &241.700    & 0.8294   &  3.9697\\ \hline
$5$ & 69.8174 &  345.043  & 3.6894 & 14.3143 & 9.5806  &236.102    & 0.8010   &  3.9157\\ \hline
$6$ & 63.5454 &  334.585  & 2.9956 & 13.8062 & 4.4398  &230.663    & 0.8010   &  3.8631\\ \hline
$7$ & 47.0467 &  324.701  & 2.5469 & 13.3262 & 4.4398  &225.381    & 0.7088   &  3.8118\\ \hline
$8$ & 41.6727 &  315.354  & 2.0430 & 12.8726 & 2.5503  &220.251    & 0.5655   &  3.7618\\ \hline
$9$ & 34.2140 &  306.510  & 1.8335 & 12.4441 & 2.5503  &215.269    & 0.5655   &  3.7130\\ \hline
$10$ & 28.7248 & 298.138  & 1.4784 & 12.0390 & 1.7127  &210.431    & 0.5078   &  3.6654\\ \hline
$11$ & 25.6050 & 290.206  & 1.3764 & 11.6560 & 1.7127  &205.734    & 0.4060   &  3.6190\\ \hline
$12$ & 21.1869 & 282.687  & 1.1178 & 11.2938 & 1.2775  &201.173   & 0.4060   &   3.5737\\ \hline
$13$ & 19.5588 & 275.554  & 1.0686 & 10.9511 & 1.2775  &196.745   & 0.3759   &   3.5296\\ \hline
$14$ & 16.5854 & 268.782  & 0.8742 & 10.6267 & 1.0185  &192.446   & 0.3004   &   3.4865\\ \hline
$15$ & 15.2815 & 262.348  & 0.8524 & 10.3195 & 1.0185  &188.272   & 0.3004   &   3.4444\\ \hline
$16$ & 13.4626 & 256.230  & 0.7020 & 10.0284 & 0.8434  &184.220   & 0.2819   &   3.4034\\ \hline
$17$ & 12.2075 & 250.408  & 0.6952 & 9.75250 & 0.8434  &180.287   & 0.2300   &   3.3633\\ \hline
$18$ & 11.0959 & 244.863  & 0.5760 & 9.49071 & 0.7113  &176.469   & 0.2300   &   3.3242\\ \hline
$19$ & 9.9938  & 239.577  & 0.5760 & 9.24220 & 0.7113  &172.762   & 0.2185   &   3.2860\\ \hline
$20$ & 9.2373  & 234.534  & 0.4815 & 9.00615 & 0.6064  &169.164   & 0.1817   &   3.2487\\ \hline
  \end{tabular}
  }\end{center}
  \end{table}

  The results in the table show that the bound in  Theorem \ref{thm:main} is far from tight for these examples.
  In fact, it may well be that the convergence rates of $\underline{f}_{\mathbf{K}}^{(r)}$ and $SA^{(r)}$ are different
  for convex $f$. We know that $SA^{(r)} - f_{\min,\bK}= \Theta(1/r)$ is the exact convergence rate for the simulated annealing
  bounds for convex $f$ (cf.\ Example \ref{ex:tight bound}), but it was speculated in \cite{DHL SIOPT} that one may in fact have
   $\underline{f}_{\mathbf{K}}^{(r)} - f_{\min,\bK}= O(1/r^2)$, even for non-convex $f$. Determining the exact convergence rate $\underline{f}_{\mathbf{K}}^{(r)}$ remains an open problem.

Finally, one should point out that it is not really meaningful to compare the computational complexities of computing the
two bounds $\underline{f}_{\mathbf{K}}^{(r)} $ and $SA^{(r)}$, as explained below.

For any polynomial $f$ and convex body $\bK$, $\underline{f}^{(r)}_{\mathbf{K}}$ may be computed by solving a generalised
 eigenvalue problem with matrices of order ${n+r \choose r}$, as long as the moments of the Lebesgue measure on $\bK$ are known.
The generalised eigenvalue computation may be done in $O\left({n+r \choose r}^3\right)$  operations; see \cite{KLLS MOR} for details.
Thus this is a polynomial-time procedure for fixed values of $r$.

For non-convex $f$, the complexity of computing $\mathop{\mathbb{E}}_{X \sim P_{f/t  }}[f(X)]$ is not known.
When $f$ is linear, it is shown in \cite{Abernethy_Hazan_2015} that $\mathop{\mathbb{E}}_{X \sim P_{rf  }}[f(X)]$ with $t = O(1/r)$ may be obtained
in $O^*\left(n^{4.5}\log(r) \right)$ oracle membership calls for $\bK$, where the $O^*(\cdot)$ notation suppresses logarithmic factors.

Since the assumptions on the available information is different for the two types of bounds, there is no simple way to compare these respective complexities.


%
%
%





\end{document}